\documentclass[a4paper,12pt]{amsart}

\usepackage{fouriernc}
\usepackage[T1]{fontenc}

\usepackage{graphicx}
\usepackage{amsmath}
\usepackage{amssymb}
\usepackage[active]{srcltx}

\usepackage{hyperref}
\usepackage{bbm}
\usepackage{enumerate}
\usepackage{tikz}

\usepackage{geometry}
\usepackage{color}
\definecolor{red}{rgb}{1,0,0}
\def\red{\color{red}}
\definecolor{green}{rgb}{0,0.8,0}
\def\gree{\color{green}}
\definecolor{blue}{rgb}{0,0,1}
\def\blu{\color{blue}}
\definecolor{black}{rgb}{0,0,0}

\definecolor{grey}{rgb}{0.333,0.333,0.333}

\definecolor{lightgrey}{rgb}{0.666,0.666,0.666}

\definecolor{white}{rgb}{1,1,1}

\newtheorem{thm}{Theorem}%
\newtheorem{prop}{Proposition}[section]%
\newtheorem{cor}{Corollary}[section]%
\newtheorem*{thm-non}{Theorem}
\theoremstyle{definition}

\theoremstyle{remark}
\newtheorem{remark}{Remark}[section] %

\theoremstyle{plain}

\numberwithin{equation}{section}

\def\HH{{\mathbb H}}

\def\NN{{\mathbb N}}

\def\RR{{\mathbb R}}

\def\ZZ{{\mathbb Z}}

\def\one{{\mathbbm{1}}}

\def\vecsigma{{\text{\boldmath$\sigma$}}}

\def\scrA{{\mathcal A}}
\def\scrB{{\mathcal B}}
\def\scrC{{\mathcal C}}
\def\scrD{{\mathcal D}}
\def\scrE{{\mathcal E}}
\def\scrF{{\mathcal F}}

\def\scrH{{\mathcal H}}

\def\scrP{{\mathcal P}}

\def\scrX{{\mathcal X}}
\def\scrY{{\mathcal Y}}

\def\e{\mathrm{e}}
\def\i{\mathrm{i}}

\def\L{\operatorname{L{}}}

\def\SL{\operatorname{SL}}

\def\SO{\operatorname{SO}}

\def\T{\operatorname{T{}}}

\def\vol{\operatorname{vol}}

\title[Extreme events for horocycle flows]{Extreme events for horocycle flows}

\author{Jens Marklof}
\address{Jens Marklof, School of Mathematics, University of Bristol, Bristol BS8 1UG, U.K.\newline \rule[0ex]{0ex}{0ex} \hspace{8pt}{\tt j.marklof@bristol.ac.uk}}
\author{Mark Pollicott}
\address{Mark Pollicott, Department of Mathematics, Zeeman Building, University of Warwick, Coventry CV4 7AL, UK 
\newline \rule[0ex]{0ex}{0ex} \hspace{8pt}{\tt masdbl@warwick.ac.uk}}
\date{29 June 2024/3 March 2025}
\thanks{JM's research supported by EPSRC grant EP/W007010/1.   MP's research supported by ERC grant  833802-resonances and EPSRC grant EP/T001674/1.
Data supporting this study are included within the article. MSC (2020): 37D40,  37D50, 11B57}

\begin{document}

\begin{abstract}
We prove extreme value laws for cusp excursions of the horocycle flow in the case of surfaces of constant negative curvature. The key idea of our approach is to study the hitting time distribution for shrinking Poincar\'e sections that have a particularly simple scaling property under the action of the geodesic flow. This extends the extreme value law of Kirsebom and Mallahi-Karai [arXiv:2209.07283] for cusp excursions for the modular surface. Here we show that the limit law can be expressed in terms of Hall's formula for the gap distribution of the Farey sequence.
\end{abstract}

\maketitle

\tableofcontents

\section{Introduction}\label{secIntro}
Important examples in the history of ergodic theory are the geodesic and horocycle flows on finite area hyperbolic surfaces, which were the basis for the pioneering works of
Hedlund and Hopf in the 1930s on minimality and ergodicity. 
Such  geodesic flows were shown to be strong mixing by Hedlund, using a geometric argument \cite{Hedlund},  and 
in 1952   by Gelfand and Fomin, using unitary representation theory \cite{FominGelfand}. 
The following year Parasjuk used  this approach  to show that  horocycle flows were also strong mixing \cite{Parasjuk}.
Later Ratner \cite{Ratner} and Moore \cite{Moore} refined these results to study the speeds of mixing of both  flows.
Other important progress came with the  classification  of horocycle invariant measures and equidistribution for horocycle flows, which was established by Furstenberg \cite{Furstenberg} (for compact surfaces) and Dani and Smillie \cite{DaniSmillie} (for non-compact surfaces).

Horocycle flows are of particular interest since they enjoy statistical features that are reminiscent of those seen in ``chaotic'' dynamics, despite the lack of sensitive dependence on initial conditions. These include the above strong mixing property, non-standard limit theorems for ergodic averages \cite{Bufetov14} and temporal limit theorems \cite{Dolgopyat18}. Particularly remarkable is the logarithm law for cusp excursions for typical horocycles due to Athreya and Margulis \cite{AthreyaMargulis} (see also \cite{AthreyaMargulis2,Kelmer12,Kelmer19,Yu17} for unipotent flows in more general settings), by analogy with results for geodesic flows by Sullivan \cite{Sullivan}. 

Kirsebom and Mallahi-Karai \cite{Kirsebom} refined the logarithm law by proving an extreme value theorem for cusp excursions on the modular surface, again by analogy with results for geodesic flows \cite{Pollicott}. One objective of the present paper is to extend their extreme value theorem to general hyperbolic surfaces and more general initial data. It should also be possible to extend the approach developed here to cusp excursions in infinite volume surfaces, using the techniques of \cite{Lutsko22}.

To state our first main result, let  $\scrX=\T^1\scrY$ be the unit tangent bundle of a non-compact hyperbolic surface $\scrY$ with finite area. Let $\varphi_t$, $h_t^+$ and $h_t^-$ be the geodesic flow, unstable horocycle flow, and stable horocycle flow on $\scrX$ at time $t$, respectively, so that
\begin{equation}\label{commute}
\varphi_t \circ h_s^+ = h_{s\exp t}^+ \circ \varphi_t ,\qquad \varphi_t \circ h_s^- = 
h_{s\exp (-t)}^- \circ \varphi_t.
\end{equation}
Denote by $\mu$ the invariant Liouville measure on $\scrX$, which we assume is normalised so that $\mu(\scrX)=1$. We denote by $d_\scrY$ the Riemannian distance on $\scrY$, which is normalised so that the geodesic flow moves points with speed one. That is, if $\pi:\scrX\to\scrY$ is the natural projection, then $d_\scrY(\pi x, \pi\circ\varphi_t x) =t$ for every $x\in\scrX$ and every sufficiently small $t>0$.  

\begin{figure}
\begin{center}
\includegraphics[width=0.7\textwidth]{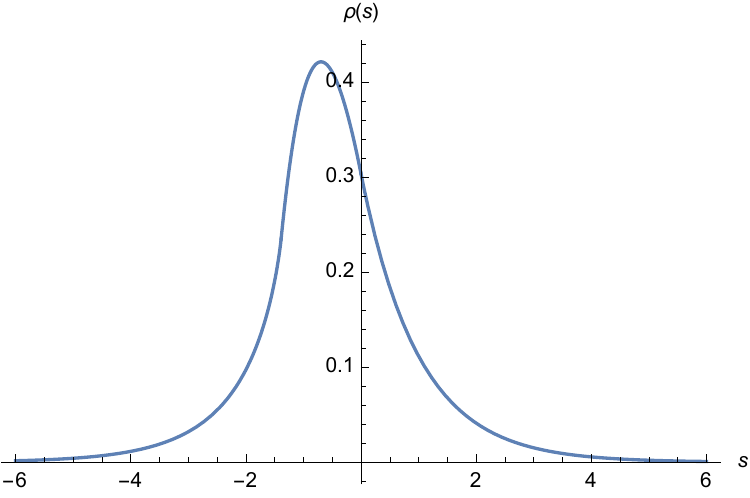}
\end{center}
\caption{The limit density for extreme cusp excursions for the modular surface, which is derived in Section \ref{secExtreme}. This is the same distribution, up to scaling and reflection, as for the logs of smallest denominators \cite{smalld2}.} \label{fig1}
\end{figure}

We then have the following generalisation of \cite{Kirsebom}.

\begin{thm}\label{mainthm}
Fix $y\in\scrY$. Let $\lambda$ be a Borel probability measure that is absolutely continuous with respect to $\mu$. Then there exists a probability density $\omega_y\in\L^1(\RR)$ and constants $C_2>C_1>0$ such that, for every $H\in\RR$, 
\begin{equation}\label{maineq}
\lim_{T \to \infty} \lambda\{ x_0\in\scrX :  \sup_{0<s\leq T}  d_\scrY(y, \pi\circ h^+_s(x_0) ) >  H + \log T  \} 
=  \int_H^\infty \omega_y(s) ds ,
\end{equation}
and
\begin{equation}\label{C1C2}
C_1\, \e^{-|s|} \leq \omega_y(s) \leq C_2\, \e^{-|s|}
\end{equation}
for all $s\in\RR$.
\end{thm}

 In particular, the extreme events in the title of this article refer to the maximum distance of the excursion of  the horocycle orbits  into the cusps.  Theorem \ref{mainthm} describes the measure of the set of points whose  orbits  up to  time $T$ have that  these values exceed $\log T$  by at least any  given fixed value $H$ as $T$ tends to infinity.

Part of the inspiration for these results comes from the case of 
geodesic flows, cf. \cite{Pollicott} and Dolgopyat's thesis \cite{Dolgopyat97}.  
There are related partial results  for more general Fuchsian groups due  to Jaerisch, Kessebomer and Stratmann \cite{Jaerisch}.

\begin{remark}
A formula for the limit density $\omega_y(s)$ is given in Section \ref{secExtreme}. Figure \ref{fig1} shows the limit density for the modular surface; see Section \ref{secExtreme} for details.
\end{remark}

\begin{remark}
The assumption that the probability measure $\lambda$ is absolutely continuous with respect to $\mu$ 
in Theorem \ref{mainthm} 
can be considerably relaxed. Key in our proof is relation \eqref{mixeq} below, which holds for a larger class of measures. These include for example probability measures $\lambda$ that are absolutely continuous with respect to the uniform measure on a fixed stable horocycle. \end{remark}

The main idea of our approach is to first establish a limit law for the hitting times of horocycles relative to a Poincar\'e section with a particularly simple scaling property (Section \ref{secSetting}), and then extend the results to the cuspidal neighbourhoods studied in \cite{Kirsebom} (Section \ref{secExtreme}). The specific Poincar\'e section we are using was constructed in \cite{Athreya12} for the modular surface (see also \cite{Marklof10} for a construction in arbitrary dimension), and in \cite{Athreya16,Athreya15,Berman23,Heersink,Kumanduri24,Sanchez,Taha19,Work20,Uyanik16} 
for general hyperbolic surfaces of finite volume, as well as more general moduli spaces of flat surfaces. The return time statistics for the horocycle flow have been studied in detail in the above papers. We provide a self-contained discussion in Sections \ref{secHitting}--\ref{secMultiple}. 
As we will see, the hitting time statistics of horocycle flows do not fall into the class of standard limit laws and are in fact close to those of integrable systems \cite{Dettmann17,Marklof10}.

In the case of the modular surface, the return times to the section are related to statistical properties of Farey sequences \cite{Athreya12,Marklof13}, which will allow us to express the extreme value law in terms of the classical Hall distribution \cite{Hall70} (see also \cite{Boca05,Kargaev97}), thus completing the partial formula found in \cite{Kirsebom}. In the general setting, the return time statistics for the above Poincar\'e sections yields the gap distribution of slopes of saddle connections on flat surfaces, which was the principal motivation for \cite{Athreya16,Athreya15,Berman23,Heersink,Kumanduri24,Sanchez,Taha19,Work20,Uyanik16}. 

The plan for the remainder of this paper is as follows. In Section \ref{secSetting} we define a shrinking family of Poincar\'e sections for the horocycle flow, and show that the hitting time process, for random initial data, converges to a limit process. The key input in the proof is the scaling property of the section and the mixing property of the geodesic flow. The limit law is then extended to cuspidal neighbourhoods by establishing their proximity to the section. This argument will require us to also track the location of the hit, and we will therefore prove limit theorems for the joint process of hitting times and impact parameters. In Section \ref{secHitting} we express the hitting time/impact parameter process in terms of a natural lattice point counting problem in $\SL(2,\RR)$, from which we can deduce additional properties of the limit process. If we are only interested in the hitting time and the ``height'' of the hit, we can further reduce the problem to understanding the distribution of the orbit of the unit vector $(0,1)$ under $\Gamma$ (acting by right multiplication). This is carried out in Section \ref{secHyperbolic} and yields an explicit formula for the average return time 
$\overline\eta_1$ to the Poincar\'e section in terms of the volume of the hyperbolic surface. Section \ref{secExample} considers the special example of the modular group $\Gamma=\SL(2,\ZZ)$, where we show that the limit laws can be expressed via the Hall distribution, and provide explicit formulas for the first hitting time. Section \ref{secMultiple} extends the setting to the joint distribution of hitting times in shrinking sections in all $K$ cusps of the surface. Section \ref{secExtreme} shows that Theorem \ref{mainthm} is now a consequence of the results obtained for hitting times, and provides a formula for the limit distribution.

\section{Hitting time statistics}\label{secSetting}

Throughout this paper we assume the surface $\scrY$ has $K$ cusps, and $\scrC_\kappa$ is a closed stable horocycle of length one around the $\kappa$th cusp, for  $\kappa=1,\ldots,K$. This means the unit tangent vectors on the horocycle point towards the cusp; in other words, for $x\in\scrC_\kappa$ we have $h_s^-(x)=h_{s+1}^-(x)$ and $\varphi_t(x)$ escapes into the cusp as $t\to\infty$. 

For $R\in\RR$, define 
\begin{equation}
\scrH_\kappa(R)
=  \{ \varphi_t(\scrC_\kappa) :  t\geq R \} ,
\end{equation}
which has the obvious scaling property
 $\varphi_t \scrH_\kappa( R) 
 = \scrH_\kappa( R+t )$. By construction, the set $\scrH_\kappa(R)$ is an embedded closed submanifold of $\scrX$ of co-dimension one, which (by the Anosov property of $\varphi_t$)  is transversal to the unstable horocycle flow and therefore defines a Poincar\'e section. We furthermore note that $\scrH_\kappa\cap\scrH_{\ell}=\emptyset$ for $\kappa\neq\ell$. We will in the following focus on the hitting times for the cusp $\kappa$.

Note that due to the ergodicity of the horocycle flow $\mu$-almost every point will eventually hit the section $\scrH_\kappa(R)$. A natural parametrisation of $\scrH_\kappa( R)$ is given by the embedding 
\begin{equation}
\iota_R : \Sigma = [0,1) \times \mathbb R_{\geq 0}\to 
\scrH_\kappa( R), \qquad (s,t) \mapsto \varphi_{t+R} \circ h_s^- (x_\kappa) 
\end{equation}
for some fixed but arbitrary choice of $x_\kappa\in\scrC_\kappa$, and $\Sigma=[0,1)\times\RR_{\geq 0}$, where $[0,1)$ parametrises $\RR/\ZZ$ with $0$ and $1$ identified.

The map $(r,s,t) \mapsto h_r^+\circ \varphi_t \circ h_s^- (x_\kappa)$ provides a parametrisation of $\scrX$ (up to a set of $\mu$-measure zero), where $(s,t)\in\Sigma$ and $0\leq r <\eta_1(s,t)$, and $\eta_1(s,t)$ is the first return time to the section $\scrH_\kappa( 0)=\iota_0(\Sigma)$. In these local coordinates, the Liouville measure reads 
\begin{equation}
d\mu(r,s,t) = \frac{1}{\overline\eta_1} \; dr \, d\nu(s,t)
\end{equation}
with
\begin{equation}\label{Kac}
d\nu(s,t) =\e^{-t} dt\, ds , \qquad 
\overline\eta_1  = \int_{\Sigma} \eta_1(s,t)\, d\nu(s,t) .
\end{equation}
This choice of normalisation ensures, in view of Kac's formula for the average return time, that $\mu(\scrX)=1$. For more background on the connection of hitting and return times, see \cite{entry}.
 
         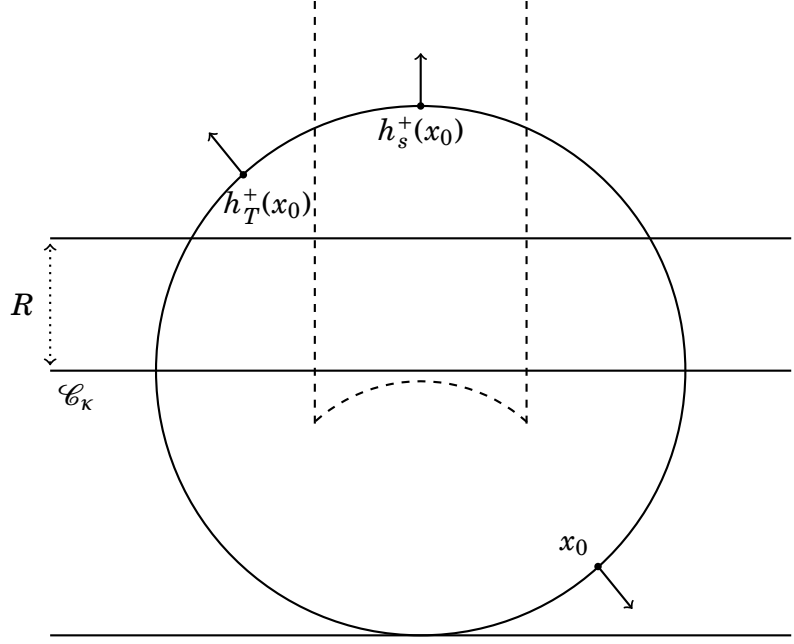
\begin{figure}
          \centerline{
    \begin{tikzpicture}[thick,scale=0.70, every node/.style={scale=1}]
\draw [black,dashed,domain=48:132] plot ({3*cos(\x)}, {3*sin(\x)-3.2});
\draw[black] (0,0) circle (5cm);
\draw[black] (-7,-5) -- (7,-5);
            \draw[black,<->, dotted] (-7,0.1) -- (-7,2.4);
               \node at (-7.5, 1.25) {$R$};
                                              \draw[fill,black] (3.35,-3.7) circle (0.05cm);
                                                            \node at (2.9,-3.3){$x_0$};
                                                             \draw[->,black] (3.35,-3.7) --(4,-4.5);
       \draw[fill,black] (0, 5) circle (0.05cm);
        \draw[->,black] (0,5) --(0,6);
        \node at (0,4.5){$h^+_s(x_0)$};
          \draw[fill,black] (-3.35,3.7) circle (0.05cm);
           \draw[->,black] (-3.35,3.7) --(-4,4.5);
             \node at (-2.9,3.1){$h^+_T(x_0)$};
           \draw[black, dashed] (-2,-1) -- (-2,7);      
                  \draw[black, dashed] (2,-1) -- (2,7);    
                  \draw[black] (-7, 2.5) -- (7,2.5);
                  \draw[black] (-7, 0) -- (7,0);
                  \node at (-6.5,-0.5){$\scrC_\kappa$};
\end{tikzpicture}
}
\caption{The unstable horocycle flow $h^+$ on $\scrX$ can be represented as a flow along horocycles in the complex upper half plane $\HH$, moving in counterclockwise direction with unit tangent vectors pointing outwards. The orbit of $x_0$ has a locally maximal excursion into the cusp at time $0 \leq s \leq T$, and hyperbolic distance from $\scrC_\kappa$ greater than $R\geq 0$.  The dashed region represents a fundamental domain in $\HH$ of the fundamental group $\Gamma$ of the surface $\scrY$.}
  \end{figure}

For $j \in \mathbb N$, denote by $\xi_j=\xi_j(x_0,R)$ the $j$th hitting time 
to $\scrH_\kappa(R)$ for the orbit of $x_0\in\scrX$ under the unstable horocycle flow $h^+$. We denote by $w_j=w_j(x_0,R)\in\Sigma$ the impact parameter of the $j$th hit in local coordinates, so that $\iota_R (w_j(x_0,R))$ is the location of the $j$th hit on $\scrH_\kappa(R)$. 

We observe the simple scaling properties \begin{equation}\label{scaling}
\xi_j(x,R) = \e^R \xi_j(\varphi_{-R} x,0), \qquad w_j(x,R) = w_j(\varphi_{-R} x,0).
\end{equation}
These will be key in our first limit law, which shows that for random initial data $x_0\in\scrX$, the sequence $(\xi_j,w_j)_{j\in\NN}$ converges to a limiting process, which is given by the hitting time process of the horocycle flow with respect to the fixed section $\scrH_\kappa(0)$.

The initial data $x_0\in\scrX$ will be chosen according to a fixed Borel probability measure $\lambda$. We say the Borel probability measure $\lambda$ is {\em admissible} if
\begin{equation}\label{mixeq}
\lim_{t\to\infty} \lambda(\varphi_t\scrE) = \mu(\scrE)
\end{equation}
for any Borel set $\scrE\subset\scrX$ with boundary of $\mu$-measure zero.
Examples of admissible measures are:
\begin{itemize}
\item Any $\lambda$ that are absolutely continuous with respect to $\mu$. This follows from the strong mixing property of the geodesic flow $\varphi$.
\item Uniform probability measures $\lambda$ supported on a  finite {\em stable} horocycle segment $\{ h_s^-(x_0) : \alpha<s<\beta\}$ with $x_0\in\scrX$ and $\alpha<\beta$ fixed, or more generally, probability measures that are absolutely continuous with respect to the uniform measure on a fixed stable horocycle. This follows from the equidistribution of translates of stable horocycles under $\varphi_t$ as $t\to-\infty$ \cite{Strombergsson}.
\end{itemize}

We denote by $\eta_{-1}(w)$ the return time with end point $\iota_0(w)\in \scrH_\kappa( 0)$, i.e., the first return time with initial condition $\iota_0(w)$ going backwards in time. 

\begin{thm}\label{cusp-first}
Assume $\lambda$ is admissible. Then, for $X>0$ and $\scrD\subset\Sigma$ a Borel set with boundary of Lebesgue measure zero, 
\begin{equation}
\lim_{R \to \infty} \lambda\{ x_0\in\scrX :  \e^{-R}\xi_1( x_0,R) >X, \; w_1(x_0,R) \in\scrD  \} \\
=  \int_X^\infty \int_\scrD \Psi(r,w) d\nu(w) dr 
\end{equation}
with $\L^1$ probability density
\begin{equation}\label{Psirw}
\Psi(r,w) = \frac{1}{\overline\eta_1} \; \one\left(0 \leq r<\eta_{-1}(w)\right) .
\end{equation}
\end{thm}

\begin{proof}
By the scaling property \eqref{scaling}, we have
\begin{equation}\label{limlaw2}
\begin{split}
&\lambda\{ x_0\in\scrX :  \e^{-R}\xi_1( x_0,R) >X, \; w_1(x_0,R) \in\scrD  \} \\
&= \lambda\{ x_0\in\scrX :  \xi_1( \varphi_{-R} x_0,0) >X, \; w_1(\varphi_{-R} x_0,0) \in\scrD  \} \\
&=\lambda(\varphi_{R}\scrE) ,
\end{split}
\end{equation}
where 
\begin{equation}
\begin{split}
\scrE & = \{ x\in\scrX :  \xi_1( x,0) >X, \; w_1(x,0) \in\scrD  \} \\
& = \iota_0 \{ (r,w) \in \RR_{>0}\times \Sigma :  X<r\leq \eta_{-1}(w),\; w\in\scrD \},
\end{split}
\end{equation}
and hence
\begin{equation}
\mu(\scrE) = \int_X^\infty \int_\scrD \Psi(r,w) d\nu(w) dr  .
\end{equation}

The right hand side of \eqref{limlaw2} converges by the assumption on $\lambda$ to $\mu(\scrE)$; recall \eqref{mixeq}. The required $\mu(\partial\scrE)=0$ follows from the assumption that $\scrD_j$ has boundary of Lebesgue measure zero, see Proposition \ref{lem-zero} below. This in turn implies that $\Psi\in\L^1(\RR_{>0}\times\Sigma,\vol_\RR\otimes\nu)$ and hence, by our normalisation, a probability density.
\end{proof}

Instead of hitting times for the Poincar\'e section $\scrH_\kappa(R)$ in the unit tangent bundle $\scrX$, we can consider the $j$th entry time $\xi^\pi_j(x_0,R)$ to its projection $\overline\scrH_\kappa(R)=\pi\scrH_\kappa(R)$ to the surface $\scrY$. The pre-image $\widehat\scrH_\kappa(R)=\pi^{-1}\overline\scrH_\kappa(R)$ is a full-dimensional set.
 The orbit of $x_0$ after entering the section at time $\xi^\pi_j(x_0,R)$ spends time $2\delta_j(x_0,R)$ (the s\'ejour time) in $\widehat\scrH_\kappa(R)$, where 
\begin{equation}\label{sejour}
\delta_j(x_0,R) = (\e^{t_j(x_0,R)}-1)^{1/2} .
\end{equation}
Here $t_j=t_j(x_0,R)$ is the second component (the vertical ``height'') of the impact parameter $w_j=(s_j,t_j)\in\Sigma$. This follows from an exercise in hyperbolic geometry, see Figure \ref{fig-ex}.
As the maximum height is achieved at half the s\'ejour time, we have the relation
\begin{equation}
\xi^\pi_j(x_0,R) = \xi_j(x_0,R) - \delta_j(x_0,R) .
\end{equation}

\begin{figure}
          \centerline{
    \begin{tikzpicture}[thick,scale=0.70, every node/.style={scale=1}]
\draw[black] (0,0) circle (5cm);
\draw [black,dashed,domain=-47:90] plot ({5.1*cos(\x)}, {5.1*sin(\x)});
\draw[black] (-7,-5) -- (7,-5);
       \draw[fill,black] (0,-5) circle (0.05cm);
            \node at (0.0, -5.3) {$0$};
            \draw[black,->] (0,-5) -- (0,6);
               \draw[black] (-7, -3.7) -- (7,-3.7);
                      \draw[fill,black] (0.0,-3.7) circle (0.05cm);
                       \node at (0.35, -3.3) {$\i a$};
                                             \draw[fill,black] (3.35,-3.7) circle (0.05cm);
                                                            \node at (2.8,-3.3){$\ell+\i a$};
                                             \draw[->,black] (3.35,-3.7) --(4,-4.5);
                  \draw[fill,black] (0, 5) circle (0.05cm);
                       \node at (0.4,5.35){$\i b$};
\end{tikzpicture}
}
\caption{An exercise in hyperbolic geometry: The figure shows a horocycle through the points $0$, $\ell+\i a$ and $\i b$, with unit tangent vectors pointing outwards, the unstable horocycle flow $h^+$ moves points in counter-clockwise motion. The objective is to calculate how long it takes to move with unit speed along the horocycle from the point $\ell+\i a$ to $\i b$, with $l\geq 0$. That is, we wish to compute the length of the horocyclic segment from $\ell+\i a$ to $\i b$. The M\"obius transformation $z\mapsto -1/z$ maps this horocycle to the horocycle $\scrC'=\{x+\i b^{-1} : x\in\RR\}$. The point $\i b$ is mapped to $\i b^{-1}$ and the point $\ell+\i a$ to $\frac{-\ell+\i a}{\ell^2+a^2}$. Since this point has to lie on $\scrC'$, we have $b^{-1}= \frac{a}{\ell^2+a^2}$, and so $\ell= (ab-a^2)^{1/2}$. Since M\"obius transformations act by isometries, the length of the horocyclic segment from  $\ell+\i a$ to $\i b$ is the same as the length of the horocyclic segment from $\i b^{-1}$ to $\frac{-\ell+\i a}{\ell^2+a^2}= -\frac{\ell}{ab}+\i b^{-1}$. The latter evaluates to $\frac{\ell}{a}=(\frac{b}{a}-1)^{1/2}$. This relation yields formula \eqref{sejour} for half of the s\'ejour time $\delta_j$ now follows by setting $a=\e^R$ and $b=\e^{R+t_j}$. 
  \label{fig-ex}}
\end{figure}
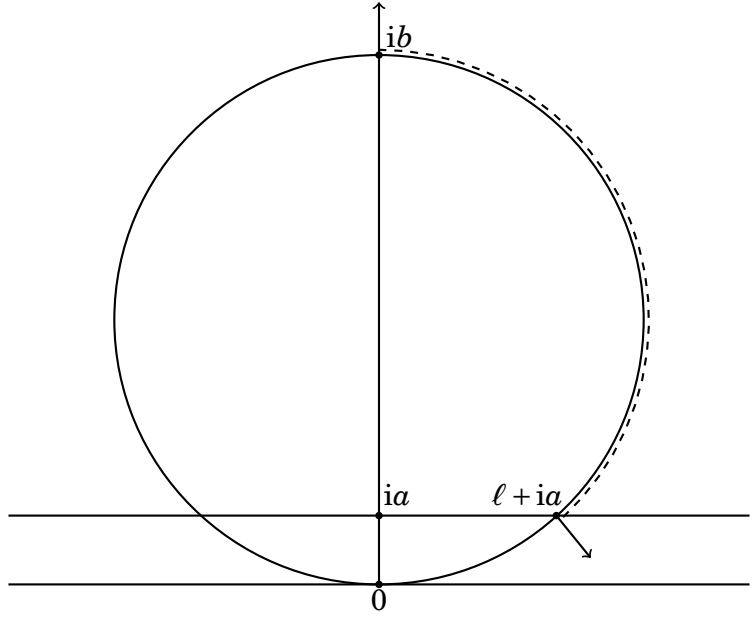

Define the probability density
\begin{equation}\label{Psidef}
\Psi(r) = \int_\Sigma  \Psi(r,w)\,  d\nu(w) .
\end{equation}

\begin{cor}\label{cusp-first-cor0}
Assume $\lambda$ is admissible.  Then, for $X>0$, 
\begin{equation}\label{psi1}
\lim_{R \to \infty} \lambda\{ x_0\in\scrX :  \e^{-R}\xi_1( x_0,R) >X \} =  \int_X^\infty  \Psi(r)  dr 
\end{equation}
and
\begin{equation}\label{psi2}
\lim_{R \to \infty} \lambda\{ x_0\in\scrX :  \e^{-R}\xi^\pi_1( x_0,R) >X \} =  \int_X^\infty \Psi(r)  dr .
\end{equation}
\end{cor}

\begin{proof}
Relation \eqref{psi1} follows from Theorem \ref{cusp-first} with $\scrD=\Sigma$. As to \eqref{psi2}, 
we have the upper bound
\begin{equation}
 \lambda\{ x_0\in\scrX :  \e^{-R}\xi^\pi_1( x_0,R) >X \}
 \leq
  \lambda\{ x_0\in\scrX :  \e^{-R}\xi_1( x_0,R) >X  \}
\end{equation}
and hence
\begin{equation}
\limsup_{R\to\infty} \lambda\{ x_0\in\scrX :  \e^{-R}\xi^\pi_1( x_0,R) >X \}  \leq 
\int_X^\infty \Psi(r)  dr .
\end{equation}
For the lower bound, for any fixed $\delta>0$, we have
\begin{equation}\label{low20}
\begin{split}
 \lambda\{ x_0\in\scrX :  \e^{-R}\xi^\pi_1( x_0,R) >X \} 
& \geq  \lambda\{ x_0\in\scrX :  \e^{-R}\xi^\pi_1( x_0,R) >X , \; \delta_1( x_0,R)\leq \delta \} \\
& \geq  \lambda\{ x_0\in\scrX :  \e^{-R}\xi_1( x_0,R) >X +\e^{-R} \delta , \; \delta_1( x_0,R)\leq \delta \} .
\end{split}
\end{equation}
Therefore Theorem \ref{cusp-first}, applied with $\scrD=\scrD_\delta=[0,1)\times [0,\log(\delta^2+1) ]$, yields
\begin{equation}
\liminf_{R\to\infty} \lambda\{ x_0\in\scrX :  \e^{-R}\xi^\pi_1( x_0,R) >X \}  \geq 
\int_X^\infty \int_{\scrD_\delta} \Psi(r,w) d\nu(w) dr  .
\end{equation}
We have here used the continuity of the limit in $X$. Since this holds for any $\delta>0$, we may take $\delta\to\infty$ in the above expression, which shows that the lim inf equals the lim sup.
\end{proof}

Theorem \ref{cusp-first} equally holds for the distribution of the $j$th hitting time and impact parameter, and indeed for the joint distribution of the first $N$ hitting times and impact parameters. The following theorem captures this, by stating that the process $$\left(\e^{-R} \xi_j(x_0,R),\;w_j(x_0,R)\right)_{j\in\NN}$$ converges (in finite-dimensional distribution) to $$\left(\xi_j(x,0),\;w_j(x,0)\right)_{j\in\NN}$$ where $x_0$ is distributed according to $\lambda$, and $x$ according to $\mu$.

\begin{thm}\label{cusp-main}
Assume $\lambda$ is admissible. Then, for $N\in\NN$, $k_1,\ldots, k_N\in\ZZ_{\geq 0}$, $0<T_1\leq \ldots\leq T_N$, and $\scrD_1,\ldots,\scrD_N\subset\Sigma$ Borel sets with boundary of Lebesgue measure zero, 
\begin{multline}
\lim_{R\to 0} \lambda\left\{ x_0\in \scrX : \#\{ j : \e^{-R} \xi_j(x_0,R)\leq T_j,\;w_j(x_0,R) \in\scrD_i \} = k_i \forall i\leq N \right\} \\
=\mu\left\{ x\in \scrX : \#\{ j : \xi_j(x,0)\leq T_j ,\; w_j(x,0) \in\scrD_i \} = k_i \forall i\leq N \right\}.
\end{multline}
\end{thm}

\begin{proof}
The strategy is identical to that in the proof of Theorem \ref{cusp-first}. In view of \eqref{scaling}, we have
\begin{equation}\label{limlaw22}
\begin{split}
&\lambda\left\{ x_0\in \scrX : \#\{ j : (\e^{-R} \xi_j(x_0,R),w_j(x_0,R)) \in[0,T_j]\times\scrD_i \} = k_i \forall i\leq N \right\} \\
&=\lambda\left\{ x_0\in \scrX : \#\{ j : (\xi_j(\varphi_{-R} x_0,0),w_j(\varphi_{-R} x_0,0)) \in[0,T_j]\times\scrD_i \} = k_i \forall i\leq N \right\} \\
&=\lambda(\varphi_{R}\scrE) ,
\end{split}
\end{equation}
where 
\begin{equation}\scrE = \left\{ x\in \scrX : \#\{ j : (\xi_j( x,0),w_j(x,0)) \in[0,T_j]\times\scrD_i \} = k_i \forall i\leq N \right\}.\end{equation}
The right hand side of \eqref{limlaw22} converges by assumption to $\mu(\scrE)$. The required $\mu(\partial\scrE)=0$ follows from the assumption that $[0,T_j]\times\scrD_j$ has boundary of Lebesgue measure zero, see Proposition \ref{lem-zero} below.
\end{proof}

\section{From hitting times to hyperbolic lattice points}\label{secHitting}

Let $\{\pm 1\}\subset\Gamma\subset\SL(2,\RR)$ be a discrete subgroup so that $\scrY=\Gamma\backslash\HH$, where $\HH$ is the complex upper half plane on which $\SL(2,\RR)$ acts by M\"obius transformations. We can further identify $\scrX=\Gamma\backslash\SL(2,\RR)$ by setting $x=\Gamma g$, and note that the choices
\begin{equation}
\varphi_t(x) = x \begin{pmatrix} \e^{t/2} & 0 \\ 0 & \e^{-t/2} \end{pmatrix}, \quad
h^+_s(x) = x \begin{pmatrix} 1 & 0 \\ -s & 1 \end{pmatrix},\quad
h^-_s(x) = x \begin{pmatrix} 1 & s \\ 0 & 1 \end{pmatrix},
\end{equation}
satisfy the commutation relations \eqref{commute}. By abuse of notation, we also denote by $\mu$ the Haar measure on $\SL(2,\RR)$, normalised so that $\mu(\scrF_\Gamma)=1$ where $\scrF_\Gamma$ is a fundamental domain of $\Gamma$ in $\SL(2,\RR)$.
The injection
\begin{equation}\label{AAAAA}
A: \RR^3\to\SL(2,\RR), \qquad (r,s,t) \mapsto \begin{pmatrix} 1 & s \\  0 & 1 \end{pmatrix} \begin{pmatrix} \e^{t/2} & 0 \\ 0 & \e^{-t/2} \end{pmatrix} \begin{pmatrix} 1 & 0 \\ r & 1 \end{pmatrix}
\end{equation}
provides a parametrisation of $\SL(2,\RR)$ up to a null set, and we have in these coordinates 
\begin{equation}
d\mu(r,s,t) = 
\overline\eta_1^{-1} \; dr \, d\nu(s,t)
\end{equation}
with
\begin{equation}
d\nu(s,t) =\e^{-t} dt\, ds 
\end{equation}
and the same 
$\overline\eta_1$
 as in \eqref{Kac}.

We may assume without loss of generality that the cusp $\kappa$ is presented in $\HH$ as a cusp at infinity of width one. This implies that $\Gamma$ contains the subgroup 
\begin{equation}
\Gamma_\infty=\left\{\begin{pmatrix} 1 & m \\ 0 & 1 \end{pmatrix} : m\in\ZZ \right\} .
\end{equation}
With this choice of coordinates, the section is represented as
\begin{equation}
\scrH_\kappa(R) = \Gamma\backslash\left\{ \Gamma \begin{pmatrix} 1 & s \\  0 & 1 \end{pmatrix} \begin{pmatrix} \e^{t/2} & 0 \\ 0 & \e^{-t/2} \end{pmatrix} : s\in[0,1), t\geq R \right\} ,
\end{equation}
and the sequence of hitting times $(\xi_j(x_0,R))_{j\in\NN}$ is precisely given by the times the lifted horocycle 
\begin{equation}\left\{ g_0 \begin{pmatrix} 1 & 0 \\ -r & 1 \end{pmatrix} : r>0 \right\}\end{equation}
intersects the disjoint union
\begin{equation}
\bigcup_{\gamma\in\Gamma} \left\{ \gamma \begin{pmatrix} 1 & s \\  0 & 1 \end{pmatrix} \begin{pmatrix} \e^{t/2} & 0 \\ 0 & \e^{-t/2} \end{pmatrix} : s\in[0,1), t\geq R \right\}
\end{equation}
where $g_0\in\SL(2,\RR)$ is any representative such that $x_0=\Gamma g_0$. Since the above union is disjoint, the hitting time $\xi_j$ determines a unique $\gamma_j\in\Gamma$ so that
\begin{equation}
g_0 \begin{pmatrix} 1 & 0 \\ -\xi_j & 1 \end{pmatrix} =
\gamma_j \begin{pmatrix} 1 & s \\  0 & 1 \end{pmatrix} \begin{pmatrix} \e^{t/2} & 0 \\ 0 & \e^{-t/2} \end{pmatrix} 
\end{equation}
with $(s+\ZZ,t-R)=w_j(x_0,R)$. On the other hand, every $\gamma\in\Gamma$ such that
\begin{equation}
\gamma g_0 \in\left\{ \begin{pmatrix} 1 & s \\  0 & 1 \end{pmatrix} \begin{pmatrix} \e^{t/2} & 0 \\ 0 & \e^{-t/2} \end{pmatrix} \begin{pmatrix} 1 & 0 \\ r & 1 \end{pmatrix} : r\in(0,T], s\in[0,1), t\geq R \right\} 
\end{equation}
yields a unique $(r,s,t)$ and $j$ such that $(\xi_j(x_0,R),w_j(x_0,R))=(r,s+\ZZ,t-R)$.

 We conclude that, for $\scrB\subset\mathbb R_{\geq0} \times \Sigma$, we have
\begin{multline}\label{pio}
\#\{ j : (\xi_j(x_0,R),w_j(x_0,R)) \in\scrB \} \\
= \# \left( \Gamma g_0 \cap \left\{ \begin{pmatrix} 1 & s \\  0 & 1 \end{pmatrix} \begin{pmatrix} \e^{t/2} & 0 \\ 0 & \e^{-t/2} \end{pmatrix} \begin{pmatrix} 1 & 0 \\ r & 1 \end{pmatrix} : (r,s,t-R)\in\scrB \right\} \right) ,
\end{multline}
which we restate in the following propositon.

\begin{prop}\label{pointlem}
For $R\in\RR$ and $\scrB\subset\RR_{>0}\times\Sigma$,
\begin{equation}
\#\{ j : (\e^{-R} \xi_j(x_0,R),w_j(x_0,R)) \in\scrB \} 
= \# \left( \Gamma g_0  \begin{pmatrix} \e^{-R/2} & 0 \\ 0 & \e^{R/2} \end{pmatrix} \cap A(\scrB)  \right) .
\end{equation}
\end{prop}

\begin{proof}
Using \eqref{pio}, we have
\begin{equation}
\begin{split}
& \#\{ j : (\e^{-R} \xi_j(x_0,R),w_j(x_0,R)) \in\scrB \} \\
& = \# \left( \Gamma g_0  \cap \left\{ \begin{pmatrix} 1 & s \\  0 & 1 \end{pmatrix} \begin{pmatrix} \e^{t/2} & 0 \\ 0 & \e^{-t/2} \end{pmatrix} \begin{pmatrix} 1 & 0 \\ r & 1 \end{pmatrix} : (r \e^{-R},s,t-R)\in\scrB \right\} \right) \\
& = \# \left( \Gamma g_0  \cap \left\{ \begin{pmatrix} 1 & s \\  0 & 1 \end{pmatrix} \begin{pmatrix} \e^{(t+R)/2} & 0 \\ 0 & \e^{-(t+R)/2} \end{pmatrix} \begin{pmatrix} 1 & 0 \\ r\e^{R} & 1 \end{pmatrix} : (r,s,t)\in\scrB \right\} \right) \\
& = \# \left( \Gamma g_0  \cap \left\{ \begin{pmatrix} 1 & s \\  0 & 1 \end{pmatrix} \begin{pmatrix} \e^{t/2} & 0 \\ 0 & \e^{-t/2} \end{pmatrix} \begin{pmatrix} 1 & 0 \\ r & 1 \end{pmatrix} \begin{pmatrix}\e^{R/2} & 0 \\ 0 & \e^{-R/2} \end{pmatrix} : (r,s,t)\in\scrB \right\} \right) \\
& = \# \left( \Gamma g_0  \begin{pmatrix} \e^{-R/2} & 0 \\ 0 & \e^{R/2} \end{pmatrix} \cap \left\{ \begin{pmatrix} 1 & s \\  0 & 1 \end{pmatrix} \begin{pmatrix} \e^{t/2} & 0 \\ 0 & \e^{-t/2} \end{pmatrix} \begin{pmatrix} 1 & 0 \\ r & 1 \end{pmatrix} : (r,s,t)\in\scrB \right\} \right) .
\end{split}
\end{equation}
\end{proof}

\begin{prop}\label{lem-zero}
For any Borel set $\scrB\subset\RR_{>0}\times\Sigma$,
\begin{equation}
\mu\left\{ x\in\scrX: \#\left\{ j : (\xi_j(x,0),w_j(x,0)) \in\scrB \right\} \geq 1 \right\} \leq \mu\left(A(\scrB)\right).
\end{equation}
\end{prop}

\begin{proof}
By Proposition \ref{pointlem} and Markov's (or Chebyshev's) inequality
\begin{equation}
\begin{split}
& \mu\left\{ x\in\scrX: \#\left\{ j : \left(\xi_j(x,R),w_j(x,R)\right) \in\scrB \right\} \geq 1 \right\} \\
& = \mu\left\{ g\in\tilde\scrF_\Gamma: \# \left( \Gamma g \cap A(\scrB) \right) \geq 1 \right\}  \\
& \leq \int_{\tilde\scrF_\Gamma} \sum_{\gamma\in\Gamma} \chi_{A(\scrB)} (\gamma g) d\mu(g),
\end{split}
\end{equation}
where $\chi_{A(\scrB)}$ is the characteristic function of the set $A(\scrB)$. Sum and integral can be combined to an integral over all of $\SL(2,\RR)$, and the result follows. 
\end{proof}

\begin{prop}\label{cusp-main-prop}
Let $N\in\NN$, $k_1,\ldots, k_N\in\ZZ_{\geq 0}$, $0<T_1\leq \ldots\leq T_N$, and $\scrD_1,\ldots,\scrD_N\subset\Sigma$ Borel sets with boundary of Lebesgue measure zero. Then
\begin{multline}
\mu\left\{ x\in \scrX : \#\left\{ j : \left(\xi_j(x,0),w_j(x,0)\right) \in(0,T_i]\times\scrD_i \right\} = k_i \forall i\leq N \right\} \\
=\mu\left\{ g\in\scrF_\Gamma : \#\left\{ \Gamma g \cap A\left((0,T_i]\times\scrD_i\right) \right\} = k_i \forall i\leq N \right\} .
\end{multline}
\end{prop}

We will now explore the case when the test set $\scrD\subset\Sigma$ only depends on the second variable, i.e., the height of the section $\Sigma$.

\section{From hyperbolic lattice points to  Euclidean point sets}\label{secHyperbolic}
Let $\SL(2,\mathbb R)$ act on $\mathbb R^2$ on the right, in the standard fashion.   
The orbit $\scrP_\Gamma=(0,1) \Gamma$ of the point $(0,1)\in\RR^2$ under the  action of the lattice $\Gamma$ is a locally finite set.
 It is in one-to-one correspondence with the left coset $\Gamma_\infty\backslash\Gamma$ via the isomorphism
\begin{equation}\label{isom}
\Gamma_\infty \backslash\Gamma \to \scrP_\Gamma, \qquad \Gamma_\infty\gamma \mapsto (0,1)\gamma.
\end{equation}
This bijection extends to the deformed set $\scrP_\Gamma g$ via
\begin{equation}\label{isom2}
\Gamma_\infty \backslash\Gamma g \to \scrP_\Gamma g, \qquad \Gamma_\infty\gamma g \mapsto (0,1)\gamma g.
\end{equation}
 It is a result of Veech that for any $\scrA\subset\RR^2$ with boundary of Lebesgue measure zero, we have
\begin{equation}\label{density}
\lim_{T\to\infty} \frac{1}{T^2} \#(\{ \gamma\in\Gamma_\infty\backslash\Gamma: (0,1) \gamma g \cap T\scrA ) = \frac{2 \vol_{\RR^2}(\scrA)}{\pi \vol_\HH(\scrF_\Gamma)} ,
\end{equation}
where $\vol_{\RR^2}$ and $\vol_\HH$ are the standard Riemannian volumes of $\RR^2$ and $\HH$, respectively.
This observation, which was key in Veech's work \cite[Section 3]{Veech89}, follows from the theory of Eisenstein series of even integral weight \cite{Kubota}, which are the generating functions for the above counting problem. (The relevant pole is a simple pole at $s=1$ for the weight zero Eisenstein series, with residue $1/\vol_\HH(\scrF_\Gamma)$. The factor $2$ reflects the fact that we are both counting $\gamma$ and $-\gamma$, whereas the Eisenstein series takes only one representative of the two.)

We furthermore have the following well known Siegel-type identity. 
Let $\tilde\scrF_\Gamma$ be any measurable fundamental domain of $\Gamma$ in $\SL(2,\RR)$, and denote by $\mu$ the Haar measure of $\SL(2,\RR)$ normalised such that $\mu(\tilde\scrF_\Gamma)=1$; it represents the normalised Liouville measure on $\scrX$, hence the same notation.

\begin{prop}\label{propSiegel}
For every measurable function $f: \RR^2\to\RR_{\geq 0}$,
\begin{equation}\label{Siegel}
\int_{\tilde\scrF_\Gamma} \bigg( \sum_{p\in\scrP_\Gamma g} f(p) \bigg)d\mu(g) = \frac{2}{\pi \vol_\HH(\scrF_\Gamma)} \int_{\RR^2} f(a) da. 
\end{equation}
\end{prop}

\begin{proof}
We can assume without loss of generality that $f$ has compact support.  Then
\begin{equation}\label{lila}
\begin{split}
\int_{\tilde\scrF_\Gamma} \bigg( \sum_{p\in\scrP_\Gamma g} f(p) \bigg)d\mu(g) 
& = \int_{\tilde\scrF_\Gamma} \bigg( \sum_{\gamma\in\Gamma_\infty\backslash\Gamma} f((0,1)\gamma g) \bigg)d\mu(g) \\
& = \int_{\tilde\scrF_{\Gamma_\infty}}  f((0,1)g) d\mu(g) .
\end{split}
\end{equation}
We can  use the Iwasawa decomposition
\begin{equation}
g = \begin{pmatrix} 1 & u \\ 0 & 1 \end{pmatrix} \begin{pmatrix} v^{1/2} & 0 \\ 0 & v^{-1/2} \end{pmatrix} \begin{pmatrix} \cos\theta & -\sin\theta \\ \sin\theta & \cos\theta \end{pmatrix}.
\end{equation}
In these coordiates
\begin{equation}
d\mu(g) = \frac{1}{\pi \vol_\HH(\scrF_\Gamma)} \frac{du\,dv\,d\theta}{v^2} ,
\end{equation}
where the normalisation ensures $\mu$ is a probability measure on the fundamental domain $\tilde\scrF_\Gamma=\scrF_\Gamma\times[0,\pi)$ (it is $\pi$ not $2\pi$ since $-1\in\Gamma$).
A fundamental domain for $\Gamma_\infty\backslash\SL(2,\RR)$ is, on the other hand given by $u\in[0,1)$, $v>0$ and $\theta\in[0,2\pi)$  (now $2\pi$ because $\Gamma_\infty$ does not contain $-1$), and so \eqref{lila} evaluates to
\begin{equation}
\frac{1}{\pi \vol_\HH(\scrF_\Gamma)} \int_0^\infty\int_0^{2\pi}  f(v^{-1/2}(\cos\theta,\sin\theta)) \frac{dv\,d\theta}{v^2},
\end{equation}
which yields the right hand side of \eqref{Siegel} after the variable substitution $r=v^{-1/2}$ and going from polar to cartesian coordinates.
\end{proof}

If we apply Proposition \ref{propSiegel} with $f=\chi_\scrA$, the characteristic function of a bounded measurable set $\scrA\subset\RR^2$, then Markov's inequality yields the useful bound
\begin{equation}\label{Siegel2}
\mu\left\{  g \in \tilde\scrF_\Gamma : \#\left( \scrP_\Gamma g \cap \scrA \right) \geq 1 \right\} \leq \frac{2 \vol_{\RR^2}(\scrA)}{\pi \vol_\HH(\scrF_\Gamma)} ,
\end{equation}
which shows in particular that it is very unlikely to find an element of $\scrP_\Gamma g$ in a small set.

\begin{prop}\label{pointprop}
For $R\in\RR$ and $\scrB\subset\RR_{>0}\times\RR_{\geq 0}$, 
\begin{equation}\label{pointeq}
\#\left\{ j : \left(\e^{-R} \xi_j(x_0,R),t_j(x_0,R)\right) \in\scrB \right\}
= \# \left( \scrP_\Gamma g_0 \begin{pmatrix} \e^{-R/2} & 0 \\ 0 & \e^{R/2} \end{pmatrix} \cap \left\{ \e^{-t/2}(1,r) : (r,t)\in\scrB \right\} \right)  .
\end{equation}
\end{prop}

\begin{proof}
In view of the bijection \eqref{isom} and Proposition \ref{pointlem},
\begin{equation}
\begin{split}
& \#\left\{ j : \left(\e^{-R} \xi_j(x_0,R),t_j(x_0,R)\right) \in\scrB \right\} \\
& = \# \left( \scrP_\Gamma g_0 \begin{pmatrix} \e^{-R/2} & 0 \\ 0 & \e^{R/2} \end{pmatrix} \cap  \left\{ (0,1) \begin{pmatrix} 1 & s \\  0 & 1 \end{pmatrix} \begin{pmatrix} \e^{t/2} & 0 \\ 0 & \e^{-t/2} \end{pmatrix} \begin{pmatrix} 1 & 0 \\ r & 1 \end{pmatrix} : (r,t)\in\scrB,\; s\in[0,1) \right\} \right) \\ 
& = \# \left( \scrP_\Gamma g_0 \begin{pmatrix} \e^{-R/2} & 0 \\ 0 & \e^{R/2} \end{pmatrix}\cap \left\{ (0,1) \begin{pmatrix} \e^{t/2} & 0 \\ 0 & \e^{-t/2} \end{pmatrix} \begin{pmatrix} 1 & 0 \\ r & 1 \end{pmatrix} : (r,t)\in\scrB \right\} \right),
\end{split}
\end{equation}
which yields \eqref{pointeq}.
\end{proof}

For $X> 0$, define the triangle $\Delta_X = \left\{ u \in\RR^2 : 0< u_1 \leq X u_2 ,\; 0< u_2 \leq 1 \right\}$.
Note that $\vol_{\RR^2}(\Delta_X)=\frac12 X$.

\begin{cor}\label{corsat}
For $X>0$,
\begin{equation}\label{pointeq002}
\#\left\{ j : \e^{-R} \xi_j(x_0,R)\leq X \right\}
=   \# \left( \scrP_\Gamma g_0 \begin{pmatrix} \e^{-R/2} & 0 \\ 0 & \e^{R/2} \end{pmatrix} \cap \Delta_X \right) .
\end{equation}
\end{cor}

\begin{proof}
In Proposition \ref{pointprop} take $\scrB=(0,X]\times \RR_{\geq 0}$. Then 
\begin{equation}
\left\{ (0,1) \begin{pmatrix} \e^{t/2} & 0 \\ 0 & \e^{-t/2} \end{pmatrix} \begin{pmatrix} 1 & 0 \\ r & 1 \end{pmatrix} : (r,t)\in\scrB \right\}
= \left\{ \e^{-t/2} (r,1) : 0<r\leq X,\; t\geq 0  \right\} .
\end{equation}
\end{proof}

The past hitting times $\xi_j$ with $j$ negative satisfy an analogue of \eqref{pointeq002} with the triangle $\Delta_{-X} = \{ u \in\RR^2 : -X u_2\leq u_1 <0  ,\; 0< u_2 \leq 1 \}$.
Again, $\vol_{\RR^2}(\Delta_{-X})=\frac12 X$.

\begin{cor}
For $r>0$, $(s,t)\in\Sigma$, 
\begin{equation}\label{rel50}
\Psi(r,s,t) = \frac{1}{\overline\eta_1} \; \one\left( \scrP_\Gamma  \begin{pmatrix} 1 & s \\  0 & 1 \end{pmatrix} \begin{pmatrix} \e^{t/2} & 0 \\ 0 & \e^{-t/2} \end{pmatrix}  \cap \Delta_{-r} = \emptyset \right) 
\end{equation}
with
\begin{equation}\label{etaformula}
\overline\eta_1  = \pi \vol_\HH(\scrF_\Gamma).
\end{equation}
\end{cor}

\begin{proof}
Relation \eqref{rel50} follows from the backward version of Corollary \ref{corsat}, since the first return time is the first hitting time for initial data on the Poincar\'e section. Proposition \ref{propSiegel} and Corollary \ref{corsat} imply that
\begin{equation}
\int_\scrX \#\left\{ j : \xi_j(x,0)\leq X \right\} d\mu(x) = \frac{2\vol_{\RR^2}(\Delta_X)}{\pi \vol_\HH(\scrF_\Gamma)}
= \frac{X}{\pi \vol_\HH(\scrF_\Gamma)} ,
\end{equation}
which means that the stationary point process $\left(\xi_j(x,0)\right)_{j\in\NN}$ (with $x\in\scrX$ random according to $\mu$) has intensity $(\pi \vol_\HH(\scrF_\Gamma))^{-1}$. 
\end{proof}

\section{Tail bounds}\label{secTail}

 In this section we will provide uniform upper and lower bounds for the density $\Psi(r)$ defined in equation \eqref{Psidef}, for all $r\geq 0$. We will assume there is only one cusp, i.e., $K=1$. The multiple cusp case is analogous and discussed in Section \ref{secMultiple}.
 
We first show that $\Psi(r)$ is constant for sufficiently small $r$. By construction, the return time of the horocycle flow to the section $\scrH_\kappa(0)$ has a strictly positive lower bound, i.e.,
\begin{equation}
\eta_1^* = \inf_{(s,t)\in\Sigma} \eta_{1}(s,t) = \inf_{(s,t)\in\Sigma} \eta_{- 1}(s,t) >0 .
\end{equation}
Hence, for $0<r<\eta_1^*$ we have,  in view of \eqref{Psirw},
\begin{equation}\label{Psirw789}
\Psi(r,s,t) = \frac{1}{\overline\eta_1} =  \frac{1}{\pi \vol_\HH(\scrF_\Gamma)}
\end{equation}
and therefore also
\begin{equation}\label{Psirw78900}
\Psi(r) =  \frac{1}{\pi \vol_\HH(\scrF_\Gamma)} .
\end{equation}

To understand the distribution of hitting times also for larger values of $r$, note that the formula for the half-s\'ejour time \eqref{sejour} yields the lower bound
\begin{equation}
\eta_{- 1}(s,t) \geq \left(\e^{t-\sigma_1} -1\right)^{1/2}
\end{equation}
for all $t\geq \sigma_1$, where $\sigma_1$ is chosen sufficiently large so that $\pi\scrH_\kappa(\sigma_1)$ is a proper cuspidal neighbourhood on the surface $\scrY$. Hence we have that 
\begin{equation}\label{Psirw789abc}
\Psi(r,s,t) \geq \frac{1}{\overline\eta_1}  \; \one\left(0\leq r< \left(\e^{t-\sigma_1} -1\right)^{1/2}\right) \; \one\left(t\geq \sigma_1\right),
\end{equation}
and therefore, by \eqref{Psidef},
\begin{equation}\label{Psirw789abcdef}
\Psi(r) \geq \frac{1}{\overline\eta_1}  \; \int_{\sigma_1}^\infty \one\left(0\leq r< (\e^{t-\sigma_1} -1)^{1/2}\right) \; \e^{-t} dt = \frac{1}{\overline\eta_1}  \; \frac{\e^{-\sigma_1}}{r^2 +1 }.
\end{equation}

Let us turn to the upper bound. We define $t'=t'(s,t)$ as the height of the pre-image under the return map on the section $\Sigma$, and $s'=s'(s,t)$ as the corresponding horizontal coordinate. Then $\eta_{-1}(s,t)=\eta_1(s',t')$ and, due to the invariance of $\nu$ under the return map, we have under the variable substitution $(s,t)\mapsto(s',t')$
\begin{equation}\label{Psidef098}
\Psi(r) = \int_\Sigma  \Psi(r,s,t)\,  d\nu(s,t) = \int_\Sigma  \Psi_+(r,s',t')\,  d\nu(s',t')
\end{equation}
where 
\begin{equation}\label{Psirw+}
\Psi_+(r,s',t') = \frac{1}{\overline\eta_1} \; \one\left(0 \leq r<\eta_{1}(s',t')\right).
\end{equation}

Now choose $\tilde\sigma_1$ sufficiently large so that 
$\pi\left(\scrH_\kappa(-\tilde\sigma_1)\right)=\scrY$. Then
\begin{equation}
\eta_{- 1}(s,t) \leq \left(\e^{t+\tilde\sigma_1} -1\right)^{1/2} + \left(\e^{t'+\tilde\sigma_1} -1\right)^{1/2}
\end{equation}
for all $(s,t)\in\Sigma$ and $(s',t')$ the image of $(s,t)$ under the first return map as above. We divide the integration in \eqref{Psidef098} into the regions $t'\leq t$ and $t\leq t'$, where
\begin{equation}\label{510}
\eta_{- 1}(s,t) \leq 2\left(\e^{t+\tilde\sigma_1} -1\right)^{1/2} \qquad (t'\leq t) ,
\end{equation}
\begin{equation}\label{511}
\eta_{1}(s',t') \leq 2 \left(\e^{t'+\tilde\sigma_1} -1\right)^{1/2} \qquad (t\leq t'),
\end{equation}
respectively. Hence, 
\begin{equation}\label{Psidef0987777}
\begin{split}
\Psi(r) & = \int_{t'<t}  \Psi(r,s,t)\,  d\nu(s,t) + \int_{t<t'}  \Psi(r,s,t)\,  d\nu(s,t) \\
& = \int_{t'<t}  \Psi(r,s,t)\,  d\nu(s,t) + \int_{t<t'}  \Psi_+(r,s',t')\,  d\nu(s',t') ,
\end{split}
\end{equation}
where in the second integration $t$ is viewed as a function of $(s',t')$.
Now, using \eqref{510}, we have for the first integral
\begin{equation}
\int_{t'<t}  \Psi(r,s,t)\,  d\nu(s,t)  \leq \frac{2}{\overline\eta_1}  \; \int_0^\infty \one\left(0\leq r< (\e^{t+\tilde\sigma_1} -1)^{1/2}\right) \; \e^{-t} dt \leq \frac{2}{\overline\eta_1}  \; \frac{\e^{\tilde\sigma_1}}{r^2 +1 } ,
\end{equation}
and the same for the second integral, using \eqref{Psirw+} and \eqref{511}. This yields
\begin{equation}
\Psi(r) \leq \frac{4}{\overline\eta_1}  \; \frac{\e^{\tilde\sigma_1}}{r^2 +1 } .
\end{equation}

Combining both upper and lower bounds, we conclude there exist constants $c_2>c_1>0$ such that for all $r\geq 0$ we have
\begin{equation}\label{5pt8}
\frac{c_1}{r^2+1} \leq \Psi(r) \leq \frac{c_2}{r^2+1}.
\end{equation}
In this general setting, the quadratic tail bound \eqref{5pt8} was first observed in \cite{Kumanduri24}.

\section{Example: the modular surface}\label{secExample}

In the case when $\Gamma$ is the modular group $\SL(2,\ZZ)$, the set $\scrP_\Gamma$ is the set of primitive lattice points $\widehat\ZZ^2=\{ (p,q) \in\ZZ^2 : \gcd(p,q)=1\}$, and the intersection of $\widehat\ZZ^2 g$ for random $x=\Gamma g$ has been studied in the context of statistics of directions of lattice points, Farey fractions and smallest denominators, see \cite{Artiles23,Athreya12,Marklof13,smalld1,smalld2} and references therein. The volume of the modular surface is $\pi/3$ and hence, by \eqref{etaformula}, we have 
$\overline\eta_1=\pi^2/3$. Furthermore note that \eqref{density} yields the well known value $6/\pi^2$ for the density of primitive lattice points.

Next, note that
\begin{equation}
\scrP_\Gamma  \begin{pmatrix} 1 & s \\  0 & 1 \end{pmatrix} \begin{pmatrix} \e^{t/2} & 0 \\ 0 & \e^{-t/2} \end{pmatrix}  \cap \Delta_{-\sigma} = \emptyset 
\end{equation}
is equivalent to 
\begin{equation}
\scrP_\Gamma  \begin{pmatrix} 1 & s \\  0 & 1 \end{pmatrix} \cap \Delta_{-\sigma,t} = \emptyset 
\end{equation}
with the triangle
$\Delta_{r,t} = \{ u \in\RR^2 : 0<  u_1 \leq r^{-1} \e^{t} u_2,\; u_2 \leq r \e^{-t/2} \}$.
A comparison with \cite[Eq.~(30)]{Marklof13} shows that
\begin{equation}
\Psi(r,t) : = \int_0^1 \Psi(r,s,t) ds = \frac{3}{\pi^2} \times
\begin{cases}
1 & \text{if $r \e^{-t/2} \leq 1$}\\
1-\e^{t/2} + r^{-1} \e^{t} & \text{if $1< r\e^{-t/2} \leq (1-\e^{-t/2})^{-1}$}\\
0 & \text{if $r\e^{-t/2} > (1-\e^{-t/2})^{-1}$.}\\
\end{cases}
\end{equation}
Integrating this against $\e^{-t}dt$ yields, via \eqref{Psidef},
\begin{equation}\label{Hall777}
\Psi(r) = \frac{3}{\pi^2} \times \begin{cases}
1 & \text{if $r\in[0,1)$}\\
-1+2r^{-1}+2r^{-1}\log r & \text{if $r\in[1,4]$}\\
-1+ 2r^{-1}+2\sqrt{\frac14-r^{-1}}-4r^{-1}\log\left(\frac12+\sqrt{\frac14-r^{-1}}\right) & \text{if $r\in[4,\infty]$,}
\end{cases}
\end{equation}
which was found by Hall \cite{Hall70} to describe the distribution of gaps in the Farey sequence. For $r\to\infty$ we have the asymptotics 
\begin{equation}
\Psi(r)\sim \frac{6}{\pi^2} r^{-2},
\end{equation}
which sharpens the general estimate \eqref{5pt8}.

We note that there are formulas analogous to the Hall distribution for more general Fuchsian groups $\Gamma$, which describe the statistics of slopes or gaps of saddle connections in flat surfaces \cite{Athreya15,Berman23,Kumanduri24,Sanchez,Taha19,Work20,Uyanik16}. By following the same steps as above, the formulas in these papers can be turned into explicit limit densities for the extreme value laws for cusp excursions on $\Gamma\backslash\HH$.

\section{Multiple-cusp sections}\label{secMultiple}
  
So far we have considered the hitting time statistics with respect to the section $\scrH_\kappa(R)$ in cusp $\kappa$. For given $\vecsigma\in\RR^K$ and $R\in\RR$, we will now consider the distribution of hitting times $(\xi_j)_{j\in\NN}$ with respect to the union of these sections
\begin{equation}
\scrH(\vecsigma, R) = \bigcup_{\kappa=1}^K \scrH_\kappa(\sigma_\kappa+R),
\end{equation}
which by construction is disjoint and transversal to the unstable horocycle flow; recall Section \ref{secSetting}. We set 
\begin{equation}
\Sigma^K(\vecsigma)=\left\{ (s,t,\kappa) : s \in [0,1),\; t\geq \sigma_\kappa,\; \kappa\in[K] \right\}
\end{equation}
with $[K]=\{1,\ldots,K\}$, and define the relevant embedding by
\begin{equation}\label{iota007}
\iota_{\vecsigma,R} : \Sigma^K(\vecsigma) \to \scrX, \qquad (s,t,\kappa) \mapsto \varphi_{t+R} \circ h_s^- (x_\kappa) 
\end{equation}
for some fixed but arbitrary choice of $x_\kappa\in\scrC_\kappa$, where $\scrC_\kappa$ is the stable horocycle of length one for cusp $\kappa$. 

For any given $\vecsigma=(\sigma_1,\ldots,\sigma_K)\in\RR^K$, The map $(r,s,t,\kappa) \mapsto h_r^+\circ \varphi_t \circ h_s^- (x_\kappa)$ provides a parametrisation of $\scrX$ (up to a set of $\mu$-measure zero), where $(s,t,\kappa)\in\Sigma^K(\vecsigma)$ and $0\leq r <\eta_1(s,t,\kappa,\vecsigma)$, the first return time to the section $\scrH(\vecsigma,0)=\iota_{\vecsigma,0}(\Sigma^K(\vecsigma))$. In these local coordinates, the Liouville measure reads as
\begin{equation}
\mu(dr,ds,dt,\{\kappa\}) = \frac{1}{\overline\eta_1(\vecsigma)} \; dr \, \nu_\vecsigma(ds,dt,\{\kappa\})
\end{equation}
with
\begin{equation}\label{Kac-multi}
\nu_\vecsigma(ds,dt,\{\kappa\}) = \e^{-t} \, dt\, ds 
\end{equation}
and
\begin{equation}\label{Kac-multi2}
\overline\eta_1(\vecsigma) = 
\int_{\Sigma^K(\vecsigma)} \eta_1(s,t,\kappa,\vecsigma)\, \nu(ds,dt,\{\kappa\})
=\sum_{\kappa=1}^K \int_0^1 \int_{\sigma_\kappa}^\infty \eta_1(s,t,\kappa,\vecsigma)\, \e^{-t}\,dt\,ds .
\end{equation}

For $j \in \mathbb N$, denote by $\xi_j=\xi_j(x_0,\vecsigma,R)$ the $j$th hitting time 
to the section $\scrH(\vecsigma,R)$ for the orbit of $x_0\in\scrX$ under the unstable horocycle flow $h^+$. We denote by $$(w_j,\kappa_j)=\left(w_j(x_0,\vecsigma,R),\kappa_j(x_0,\vecsigma,R)\right)\in\Sigma^K(\vecsigma)$$ the impact parameter of the $j$th hit in local coordinates, so that $\iota_{\vecsigma,R} (w_j,\kappa_j)$ is the location of the $j$th hit on $\scrH(\vecsigma,R)$. 

In this setting Theorem \ref{cusp-first} generalises to the following multi-cusp version. As before, we denote by $\eta_{-1}(w,\kappa,\vecsigma)$ the first backwards return time to the section $\scrH(\vecsigma,0)$. 

\begin{thm}\label{cusp-multi}
Assume $\lambda$ is admissible, and fix $\vecsigma\in\RR^K$. Then, for $X>0$ and $\scrD\subset\Sigma^K(\vecsigma)$ a Borel set with boundary of Lebesgue measure zero, 
\begin{multline}
\lim_{R \to \infty} \lambda\left\{ x_0\in\scrX :  \e^{-R}\xi_1( x_0,\vecsigma,R) >X, \; \left(w_1(x_0,\vecsigma,R),\kappa_1(x_0,\vecsigma,R)\right) \in\scrD  \right\} \\
=   \int_X^\infty \int_\scrD \Psi_\vecsigma(r,w,\kappa)\, \nu(dw,\{\kappa\}) \, dr 
\end{multline}
with $\L^1$ probability density
\begin{equation}
\Psi_\vecsigma(r,w,\kappa) = \frac{1}{\overline\eta_1(\vecsigma)} \; \one\left(0 \leq r<\eta_{-1}(w,\kappa,\vecsigma)\right) .
\end{equation}
\end{thm}

\begin{proof}
The proof is the same as for Theorem \ref{cusp-first}, rescaling and \eqref{mixeq} shows that we have convergence to the limit 
\begin{equation}
\mu\left\{ x\in\scrX :  \xi_1( x,\vecsigma) >X, \; \left(w_1(x,\vecsigma,0),\kappa_1(x,\vecsigma,0)\right) \in\scrD  \right\} .
\end{equation}
The final formula then follows as before.
\end{proof}

Theorem \ref{cusp-multi} implies the following multi-cusp extension of Corollary \ref{cusp-first-cor0}. We now consider the $j$th entry time $\xi^\pi_j(x_0,,\vecsigma,R)$ to the projection of the total section
\begin{equation}
\overline\scrH(\vecsigma,R) = \bigcup_{\kappa=1}^K \overline\scrH_\kappa(\sigma_\kappa + R)
\end{equation}
to the surface $\scrY$, where $\overline\scrH(\vecsigma,R)= \pi\scrH(\vecsigma,R)$ and $\overline\scrH_\kappa(\sigma_\kappa)=\pi\scrH_\kappa(\sigma_\kappa)$.

Define the probability density
\begin{equation}\label{Psidef-multi}
\Psi_\vecsigma(r) = \int_{\Sigma^K(\vecsigma)}  \Psi_\vecsigma(r,w,\kappa)\,  \nu_\vecsigma(dw,\{\kappa\}) 
= \sum_{\kappa=1}^K \int_0^1 \int_{\sigma_\kappa}^\infty \Psi_\vecsigma(r,s,t,\kappa)\, \e^{-t}\,dt\,ds.
\end{equation}

\begin{cor}\label{cusp-first-cor0-multi}
Assume $\lambda$ is admissible.  Then, for $X>0$, 
\begin{equation}\label{psi1-multi}
\lim_{R \to \infty} \lambda\left\{ x_0\in\scrX :  \e^{-R}\xi_1\left( x_0,\vecsigma,R\right) >X \right\} =  \int_X^\infty  \Psi_\vecsigma(r)  dr 
\end{equation}
and
\begin{equation}\label{psi2-multi}
\lim_{R \to \infty} \lambda\left\{ x_0\in\scrX :  \e^{-R}\xi^\pi_1\left( x_0,\vecsigma,R\right) >X \right\} =  \int_X^\infty \Psi_\vecsigma(r)  dr .
\end{equation}
\end{cor}

\begin{proof}
The proof is analogous to that of Corollary \ref{cusp-first-cor0}.
\end{proof}

The multi-cusp version of Theorem \ref{cusp-main} is the following.  
  
\begin{thm}\label{cusp-main-multi}
Assume $\lambda$ is admissible. Then, for $N\in\NN$, $k_1,\ldots, k_N\in\ZZ_{\geq 0}$, $0<T_1\leq \ldots\leq T_N$, and $\scrD_1,\ldots,\scrD_N\subset\Sigma^K(\vecsigma)$ Borel sets with boundary of Lebesgue measure zero, 
\begin{multline}
\lim_{R\to 0} \lambda\left\{ x_0\in \scrX : \#\left\{ j : \e^{-R} \xi_j(x_0,\vecsigma,R)\leq T_j,\;\left(w_j(x_0,\vecsigma,R),\kappa_j(x_0,\vecsigma,R)\right) \in\scrD_i \right\} = k_i \forall i\leq N \right\} \\
=\mu\left\{ x\in \scrX : \#\left\{ j : \xi_j(x,\vecsigma,0)\leq T_j ,\; \left(w_j(x,\vecsigma,0),\kappa_j(x,\vecsigma,0)\right) \in\scrD_i \right\} = k_i \forall i\leq N \right\}.
\end{multline}
\end{thm}

We now explain how the above setting can be translated to a hyperbolic lattice point problem, in  analogy with Section \ref{secHitting}.
Choose $g_\kappa\in\SL(2,\RR)$ so that it maps the cusp $\kappa$ to the cusp at infinity of width one. This is equivalent to saying that the stabilizer $\Gamma_\kappa$ of $\kappa$ is 
\begin{equation}\label{stub}
\Gamma_\kappa=g_\kappa^{-1}  \Gamma_\infty g_\kappa
= g_\kappa^{-1} \left\{\begin{pmatrix} 1 & m \\ 0 & 1 \end{pmatrix} : m\in\ZZ \right\} g_\kappa,
\end{equation}
and the section in the cusp $\kappa$ is represented as
\begin{equation}
\scrH_\kappa(\sigma_\kappa) = \Gamma\backslash \left\{ \Gamma g_\kappa^{-1} \begin{pmatrix} 1 & s \\  0 & 1 \end{pmatrix} \begin{pmatrix} \e^{t/2} & 0 \\ 0 & \e^{-t/2} \end{pmatrix} : s\in[0,1), t\geq \sigma_\kappa \right\} .
\end{equation}
The same steps leading to Proposition \ref{pointlem} give us the following.
For $A$ as in \eqref{AAAAA} we define
$A^*(r,s,t,\kappa) = (A(r,s,t),\kappa)$.

\begin{prop}\label{pointlem-multi}
For $\vecsigma\in\RR^K$, $R\in\RR$ and $\scrB\subset\RR_{>0}\times\Sigma^K(\vecsigma)$,
\begin{multline}
\#\left\{ j : \left(\e^{-R} \xi_j\left(x_0,\vecsigma,R\right),w_j\left(x_0,\vecsigma,R\right),\kappa_j\left(x_0,\vecsigma,R\right)\right) \in\scrB \right\} \\
= \# \left\{ (\gamma,\kappa)\in\Gamma\times[K] : \left(  g_\kappa \gamma g_0  \begin{pmatrix} \e^{-R/2} & 0 \\ 0 & \e^{R/2} \end{pmatrix},\kappa\right) \in A^*(\scrB)  \right\} .
\end{multline}
\end{prop}

This proposition then implies the following multi-cusp generalisation of Proposition \ref{cusp-main-prop}, which provides a formula for the limit distribution in Theorem \ref{cusp-main} in terms of hyperbolic lattice points. 

\begin{prop}\label{cusp-main-prop-multi}
Let $\vecsigma\in\RR^K$, $N\in\NN$, $k_1,\ldots, k_N\in\ZZ_{\geq 0}$, $0<T_1\leq \ldots\leq T_N$, and $\scrD_1,\ldots,\scrD_N\subset\Sigma^K(\vecsigma)$ Borel sets with boundary of Lebesgue measure zero. Then
\begin{multline}
\mu\left\{ x\in \scrX : \#\left\{ j : 
\left(\xi_j(x,\vecsigma,0),w_j(x,\vecsigma,0),\kappa_j(x,\vecsigma,0)\right) \in(0,T_i]\times\scrD_i \right\} = k_i \forall i\leq N \right\} \\
=\mu\left\{ g\in\scrF_\Gamma : \#\left\{ \left(\gamma,\kappa\right)\in\Gamma\times[K] : \left( g_\kappa \gamma g , \kappa\right) \in A^*\left((0,T_i]\times\scrD_i\right) \right\} = k_i \forall i\leq N \right\} .
\end{multline}
\end{prop}

Let us now extend the discussion of Section \ref{secHyperbolic} on the distribution of the height of the hits in the section $\Sigma$, ignoring the horizontal parameter $s\in[0,1)$, to the multi-cusp setting. It is convenient to set
\begin{equation}
\overline\Sigma^K(\vecsigma)=\left\{ (t,\kappa) :  t\geq \sigma_\kappa,\; \kappa\in[K] \right\} .
\end{equation}

For each cusp $\kappa$, the analogue of the isomorphism \eqref{isom} is
\begin{equation}\label{isom-multi}
\Gamma_\kappa \backslash \Gamma \to \scrP_\Gamma^\kappa, \qquad \Gamma_\kappa  \gamma \mapsto (0,1)g_\kappa \gamma,
\end{equation}
where we note that, in view of \eqref{stub}, we have $(0,1)g_\kappa \gamma=(0,1) g_\kappa$ if and only if $\gamma\in\Gamma_\kappa$.
We extend the bijection to the deformed set $\scrP_\Gamma^\kappa g$ via
\begin{equation}\label{isom2-multi}
\Gamma_\kappa \backslash \Gamma g \to \scrP_\Gamma^\kappa g, \qquad \Gamma_\kappa  \gamma g \mapsto (0,1)g_\kappa \gamma g.
\end{equation}
Relation \eqref{density} then implies
\begin{equation}\label{density-multi}
\lim_{T\to\infty} \frac{1}{T^2} \#(\{ \gamma\in\Gamma_\kappa\backslash\Gamma: (0,1) g_\kappa \gamma g \in T\scrA ) = \frac{2 \vol_{\RR^2}(\scrA)}{\pi \vol_\HH(\scrF_\Gamma)} .
\end{equation}
since
$$
\#(\{ \gamma\in\Gamma_\kappa\backslash\Gamma: (0,1) g_\kappa \gamma g \in T\scrA )
=
\#(\{ \gamma\in\Gamma_\infty\backslash g_\kappa \Gamma g_\kappa^{-1} : (0,1) \gamma  \in T\scrA\; (g_\kappa g)^{-1})
$$
and $\vol_\HH(\scrF_{g_\kappa \Gamma g_\kappa^{-1}})=\vol_\HH(\scrF_\Gamma)$ (because $g_\kappa$ acts on $\HH$ by isometries) and $\vol_{\RR^2}(\scrA (g_\kappa g)^{-1})=\vol_{\RR^2}(\scrA)$ (because $g_\kappa g$ has determinant one).

Proposition \ref{propSiegel} can be restated as follows.

\begin{prop}\label{propSiegel-multi}
For every measurable function $f:\RR^2\to\RR_{\geq 0}$,
\begin{equation}\label{Siegel-multi}
\int_{\tilde\scrF_\Gamma} \bigg( \sum_{p\in\scrP^\kappa_\Gamma g} f(p) \bigg)d\mu(g) = \frac{2}{\pi \vol_\HH(\scrF_\Gamma)} \int_{\RR^2} f(a) da. 
\end{equation}
\end{prop}

\begin{proof} 
In Section \ref{secHyperbolic} we assumed without loss of generality that $\kappa=\infty$, so the statement is in fact the same as Proposition \ref{propSiegel}. (It is nevertheless instructive to repeat the proof in this setting, using the unimodularity of Haar measure $\mu$.)
\end{proof}

For $\scrB\subset\RR_{>0}\times\overline\Sigma^K(\vecsigma)$ put
$\scrB^\kappa = \left\{ (r,t) \in \RR_{>0} \times [\sigma_\kappa,\infty) : (r,t,\kappa)\in\scrB  \right\}$.

\begin{prop}\label{pointprop-multi}
For $\vecsigma\in\RR^K$, $R\in\RR$ and $\scrB\subset\RR_{>0}\times\overline\Sigma^K(\vecsigma)$, 
\begin{multline}\label{pointeq-multi}
\#\left\{ j : \left(\e^{-R} \xi_j(x_0,\vecsigma,R),w_j(x_0,\vecsigma,R),\kappa_j(x_0,\vecsigma,R)\right)  \in\scrB \right\}\\
= \sum_{\kappa=1}^K \# \left\{ \scrP^\kappa_\Gamma g_0 \begin{pmatrix} \e^{-R/2} & 0 \\ 0 & \e^{R/2} \end{pmatrix} \cap \left\{ \e^{-t/2}(1,r) : (r,t)\in\scrB^\kappa \right\} \right\}  .
\end{multline}
\end{prop}

\begin{proof} 
The proof is identical to the proof of Proposition \ref{pointprop}.
\end{proof}

 \begin{cor}\label{corsat-multi}
For $X>0$,
\begin{equation}\label{pointeq002-multi}
\#\{ j : \e^{-R} \xi_j(x_0,\vecsigma,R)\leq X \}
=   \sum_{\kappa=1}^K\# \left( \e^{\sigma_\kappa/2} \scrP^\kappa_\Gamma g_0 \begin{pmatrix} \e^{-R/2} & 0 \\ 0 & \e^{R/2} \end{pmatrix} \cap \Delta_X \right) .
\end{equation}
\end{cor}

\begin{proof}
In Proposition \ref{pointprop-multi} take $\scrB=(0,X]\times\overline\Sigma^K(\vecsigma)$. Then $\scrB^\kappa = \RR_{>0} \times [\sigma_\kappa,\infty)$ and we have
\begin{equation}
\left\{ (0,1) \begin{pmatrix} \e^{t/2} & 0 \\ 0 & \e^{-t/2} \end{pmatrix} \begin{pmatrix} 1 & 0 \\ r & 1 \end{pmatrix} : (r,t)\in\scrB^\kappa \right\}
= \left\{ \e^{-t/2} (r,1) : 0<r\leq X,\; t\geq \sigma_\kappa  \right\} .
\end{equation}
\end{proof}

\begin{cor}
For $(r,s,t,\kappa)\in\RR_{>0}\times\overline\Sigma^K(\vecsigma)$, 
\begin{equation}\label{rel50-multi}
\Psi_\vecsigma(r,s,t,\kappa) = \frac{1}{\overline\eta_1(\vecsigma)} \; 
\prod_{\ell=1}^K\one\left( \e^{\sigma_{\ell}/2} \scrP^{\ell}_\Gamma  g_\kappa
\begin{pmatrix} 1 & s \\  0 & 1 \end{pmatrix} \begin{pmatrix} \e^{t/2} & 0 \\ 0 & \e^{-t/2} \end{pmatrix} 
\cap  \Delta_{-r} = \emptyset \right) 
\end{equation}
with
\begin{equation}\label{etaformula-multi}
\overline\eta_1(\vecsigma) = \pi \vol_\HH(\scrF_\Gamma) \bigg( \sum_{\kappa=1}^{K} \e^{-\sigma_\kappa} \bigg)^{-1}.
\end{equation}
\end{cor}

\begin{proof}
Relation \eqref{rel50-multi} follows from Corollary \ref{corsat-multi}, since the first return time (to the section $\scrH_\kappa(\vecsigma,0)$) is the first hitting time (to the section $\scrH_\kappa(\vecsigma,0)$) for initial data 
\begin{equation}
g_0 = g_\kappa
\begin{pmatrix} 1 & s \\  0 & 1 \end{pmatrix} \begin{pmatrix} \e^{t/2} & 0 \\ 0 & \e^{-t/2} \end{pmatrix} 
\end{equation}
on the section $\scrH_\kappa(\sigma_\kappa)$, with $s\in[0,1)$, $t\geq \sigma_\kappa$.

Proposition \ref{propSiegel-multi} and Corollary \ref{corsat-multi} imply that
\begin{equation}
\int_\scrX \#\{ j : \xi_j(x,\vecsigma,0)\leq X \}\, d\mu(x) = \sum_{\kappa=1}^{K} \frac{2\vol_{\RR^2}(\e^{-\sigma_\kappa/2}\Delta_X)}{\pi \vol_\HH(\scrF_\Gamma)} 
= \frac{X}{\pi \vol_\HH(\scrF_\Gamma)} \sum_{\kappa=1}^{K} \e^{-\sigma_\kappa},
\end{equation}
which means that the stationary point process $(\xi_j(x,\vecsigma,0))_{j\in\NN}$ (with $x\in\scrX$ random according to $\mu$) has intensity 
$$
\frac{1}{\pi \vol_\HH(\scrF_\Gamma)} \sum_{\kappa=1}^{K} \e^{-\sigma_\kappa}.
$$
This completes the proof.
\end{proof}

We finally turn to the tail estimates of Section \ref{secTail}. Again, we have positive lower bound for the return time,
\begin{equation}
\eta_1^*(\kappa,\vecsigma) = \inf_{w\in\Sigma} \eta_{1}(w,\kappa,\vecsigma) = \inf_{w\in\Sigma} \eta_{- 1}(w,\kappa,\vecsigma) >0 .
\end{equation}
For $0<r<\eta_1^*(\kappa,\vecsigma)$ we therefore have
\begin{equation}\label{Psirw789zz}
\Psi_\vecsigma(r,w,\kappa) = \frac{1}{\overline\eta_1(\vecsigma)}  ,
\end{equation}
and, for $0<r<\min_\kappa \eta_1^*(\kappa,\vecsigma)$,
\begin{equation}\label{Psirw78900zz}
\Psi_\vecsigma(r) =  \frac{1}{\overline\eta_1(\vecsigma)} \sum_{\kappa=1}^K \e^{-\sigma_\kappa}  = \frac{1}{\pi \vol_\HH(\scrF_\Gamma)} .
\end{equation}
The proof of the quadratic tail bound for large $r$ is the same as for \eqref{5pt8}, and we have
\begin{equation}\label{5pt8multi}
\frac{c_1}{r^2+1} \leq \Psi_\vecsigma(r) \leq \frac{c_2}{r^2+1},
\end{equation}
for $r\geq 0$ and constants $c_2>c_1>0$. Thus we recover the quadratic tails observed in \cite{Kumanduri24}.

\section{Extreme value laws}\label{secExtreme}

We start by discussing a variant of Theorem \ref{mainthm} in a setting closer to our previous discussion of hitting times. 

 In each cusp the height of the horocycle excursion is relative to a closed horocycle of unit length.  
Thus with surfaces  having cusps of different widths the definition of the height of horocycle excursion changes accordingly.
Using the standard parametrisation, via \eqref{iota007}, of the section $\scrH(\vecsigma,0)$ by $(s,t,\kappa)\in\Sigma^K(\vecsigma)$, we define the marked height of a point $x\in\scrH(\vecsigma,0)$ by $(\tau(x),\kappa(x))$ as the second (vertical) coordinate $t$ and cusp $\kappa$. Note that the marked height is independent of the choice of $\vecsigma$, as long as $x$ is in the section. We denote the relative height by $\widetilde\tau(x,\vecsigma)=\tau(x)-\sigma_{\kappa(x)}$. We do the same for the projected section
\begin{equation}
\overline\scrH(\vecsigma,0) = \bigcup_{\kappa=1}^K \overline\scrH_\kappa(\sigma_\kappa),
\qquad \overline\scrH_\kappa(\sigma_\kappa) =\pi \scrH_\kappa(\sigma_\kappa),
\end{equation}
assuming here that $\sigma_\kappa$ are sufficiently large so that the map 
$\pi\circ\iota_{\vecsigma,0}: \Sigma^K(\vecsigma) \to \scrY$ defines an embedding with image $\overline\scrH(\vecsigma,0)$, with $\iota_{\vecsigma,0}$ as defined in \eqref{iota007}.

Every $y\in\overline\scrH(\vecsigma,0)$ then has a unique pre-image $\{x\}=\pi^{-1}(y)\cap\scrH(\vecsigma,0)$, and we define the marked height $(\tau(y),\kappa(y))$ and relative height $\widetilde\tau(y,\vecsigma)=\tau(y)-\sigma_{\kappa(y)}$ of such a $y$ as the height and relative height of $x$, respectively.
Geometrically, $\widetilde\tau(y,\vecsigma)$ is the Riemannian distance from the horocycle $\varphi_{\sigma_\kappa}(\scrC_\kappa)$ (the boundary of $\overline\scrH_\kappa(\sigma_\kappa)$) to $y\in\overline\scrH_\kappa(\sigma_\kappa)$.

The following limit law provides the distribution of extremal heights of cusp excursions of the horocyclic orbit $\{ h_s^+(x_0) : 0 \leq s \leq T\}$ with random initial data $x_0\in\scrX$. We will consider both extreme excursions on the surface and extreme locally maximal excursions, which occur at time $\xi_j=\xi_j(x_0,\vecsigma,0)$ when the trajectory hits the section $\scrH(\vecsigma,0)$. 
As we will see below, Theorem \ref{mainthm} is a consequence of the following statement.
  
\begin{thm}\label{cusp-first-thm-multi}
Assume $\lambda$ is admissible.  Then, for $H\in\RR$, $\vecsigma\in\RR^K$, 
\begin{equation}\label{h-limit}
\lim_{T \to \infty} \lambda\left\{ x_0\in\scrX :  \sup_{0\leq s \leq T} \widetilde\tau(\pi\circ  h_s^+(x_0),\vecsigma)  >  H + \log T \right\} 
=  \int_H^\infty \rho_\vecsigma(s) ds 
\end{equation}
and
\begin{equation}\label{h-limit2}
\lim_{T \to \infty} \lambda\left\{ x_0\in\scrX :  \sup_{0\leq \xi_j \leq T} \widetilde\tau(h_{\xi_j}^+(x_0),\vecsigma)  >  H + \log T \right\} 
=  \int_H^\infty \rho_\vecsigma(s) ds 
\end{equation}
with the $\L^1$ probability density
\begin{equation}\label{cusp-first-cor-den}
\rho_\vecsigma(s) = \e^{-s} \Psi_\vecsigma(\e^{-s}) .
\end{equation}
\end{thm}

\begin{proof}
The critical observation is that
\begin{equation}
\sup_{0\leq s \leq T} \widetilde\tau(\pi\circ  h_s^+(x_0),\vecsigma)  \leq  R
\end{equation}
if and only if 
\begin{equation}
\xi^\pi_1( x_0,\vecsigma,R) >T .
\end{equation}
 Hence Eq.~\eqref{psi2-multi} of Corollary \ref{cusp-first-cor0-multi} yields, for $T=\e^R X$ and $X=\e^{-H}$, 
\begin{equation}
\lim_{T \to \infty} \lambda\left\{ x_0\in\scrX :  \sup_{0\leq s \leq T} \widetilde\tau(\pi\circ  h_s^+(x_0),\vecsigma)  \leq  H + \log T \right\} 
=  \int_{\e^{-H}}^\infty \Psi_\vecsigma(s) ds 
\end{equation}
i.e.,
\begin{equation}
\lim_{T \to \infty} \lambda\left\{ x_0\in\scrX :  \sup_{0\leq s \leq T} \widetilde\tau(\pi\circ  h_s^+(x_0),\vecsigma)  >  H + \log T \right\} 
=  1-\int_{\e^{-H}}^\infty \Psi_\vecsigma(s) ds 
\end{equation}
and the formula for the density follows by differentiation. The second claim \eqref{h-limit2} follows by the same argument, using instead Eq.~\eqref{psi1-multi}.
\end{proof}

Note that the tail bounds \eqref{Psirw789zz} and \eqref{5pt8multi} imply that
\begin{equation}\label{5pt8multirho00}
\rho_\vecsigma(s) = \frac{\e^{-s}}{\pi \vol_\HH(\scrF_\Gamma)} 
\end{equation}
for sufficiently large $s>0$, and 
\begin{equation}\label{5pt8multirho}
c_1\e^{-|s|} \leq \rho_\vecsigma(s) \leq c_2 \e^{-|s|}
\end{equation}
for $s\leq 0$.

Relation \eqref{h-limit} generalises the extreme value law for the horocycle flow on the modular surface by Kirsebom and Mallahi-Karai \cite{Kirsebom}. The modular surface has only one cusp, and we may take $\sigma=0$. 
In view of \eqref{Hall777}, the limit density $\rho_0(s) =\e^{-s} \Psi(\e^{-s})= \rho(s)$ in Theorem \ref{cusp-first-thm-multi} is
\begin{equation}\label{rho-ex}
\rho(s) = \frac{3}{\pi^2} \times
\begin{cases}
-\e^{-s}+ 2+2\e^{-s}\sqrt{\frac14-\e^s}-4\log\left(\frac12+\sqrt{\frac14-\e^s}\right) & \text{if $s\in(-\infty,-2\log 2]$} \\
-\e^{-s}+2-2 s  & \text{if $s\in[-2\log 2,0]$}\\
\e^{-s} & \text{if $s\in[0,\infty)$.}
\end{cases}
\end{equation}
and for $s\to-\infty$,
\begin{equation}
\rho(s) \sim\frac{6}{\pi^2} \e^{-|s|}  .
\end{equation}
Note the similarity with the distribution of the logs of smallest denominators \cite{smalld2}, whose density is $\eta_{\log}(s)=2\rho(-2s)$. This of course is not accidental. Indeed, smallest denominators correspond to locally maximal cusp excursions of closed horocycles, cf.~\cite{smalld1}. In particular, \cite{smalld2} gives the following formula for the expected value
\begin{equation}
\int_\RR s\, \rho(s) \,ds = - 2 \int_\RR s\, \eta_{\log}(s) \,ds = 1-\frac{12}{\pi^2}\, \zeta (3) \approx -0.46153 .
\end{equation}
Similar explicit formulas hold for higher moments. 

We note that \cite{Kirsebom} uses as a distance measure the log of the alpha function $\alpha_1$, which represents the inverse length of the shortest non-zero lattice vector in the lattice $\ZZ^2 g$,  
\begin{equation}
\alpha_1(x) = \max_{v \in\ZZ^2 g\setminus\{0\}} \frac{1}{\| v\|} ,
\end{equation}
with $g$ such that $x=\SL(2,\ZZ) g$. Here $\| \,\cdot\,\|$ is the Euclidean norm, and thus $\alpha_1$ is invariant under the right action of $\SO(2)$, hence may be viewed as a function on the modular surface. A short argument shows that for $y\in\overline\scrH_1(0)$ we have $2\log \alpha_1(y) = \widetilde\tau(y)$. The limit density in \cite{Kirsebom} is related to ours by $\rho_{\text{KMK}}(r)=2\rho(2r)$.

\begin{figure}
\begin{center}
    \begin{tikzpicture}[thick,scale=0.70, every node/.style={scale=1}]
          \draw[fill,yellow] (5,-2)--(2,10)--(5,10);
    \draw[blue] (0,0) -- (8,0);
    \draw[black,dashed] (0,0) -- (8,0);
       \draw[black,dashed] (0,0) -- (0,12);
          \draw[black,dashed] (8,0) -- (8,12);
          \draw[blue] (0,10) -- (2,10);
                  \draw[red] (5,10) -- (2,10);
                  \draw[blue,dashed] (5,10) -- (2,10);
                    \draw[blue] (5,10) -- (8,10);
          \draw[green] (2,0) -- (2,10);
         \draw[black] (2,10) -- (5,-2);
          \draw[red] (5,10) -- (5,-2);
           \draw[red,dashed] (5.1,0) -- (5.1,-2);
\filldraw[black] (2,10) circle (3pt);
\node at (2,10.5) {$y'$};
\filldraw[black] (5,-2) circle (3pt);
\node at (5,-2.5) {$y$}; 
\node at (5.6,-1.15) {\red $\ell_y$}; 
\node at (9.5,11.5) {$\overline\scrH_\kappa(\sigma_\kappa)$};
\node at (8.8,10.5) {\rm\blu horocycle of length $e^{-\tau(y')}$};
\node at (8.7,-0.5) {\rm\blu horocycle of length $e^{-\sigma_\kappa}$};
\node at (3.7,4.8) {\rotatebox{284}{geodesic of length $d(y,y')$}}; 
\node at (5.5,5) {\rotatebox{270}{\red geodesic of length $\ell_y+\tilde \tau(y',\sigma)$}}; 
\node at (1,6) {\gree $\tilde \tau(y',\sigma)$}; 
    \end{tikzpicture}
\end{center}
\caption{Schematic illustration for the triangle inequality leading to \eqref{tri}. The relevant triangle is highlighted in yellow.} \label{fig123}
\end{figure}
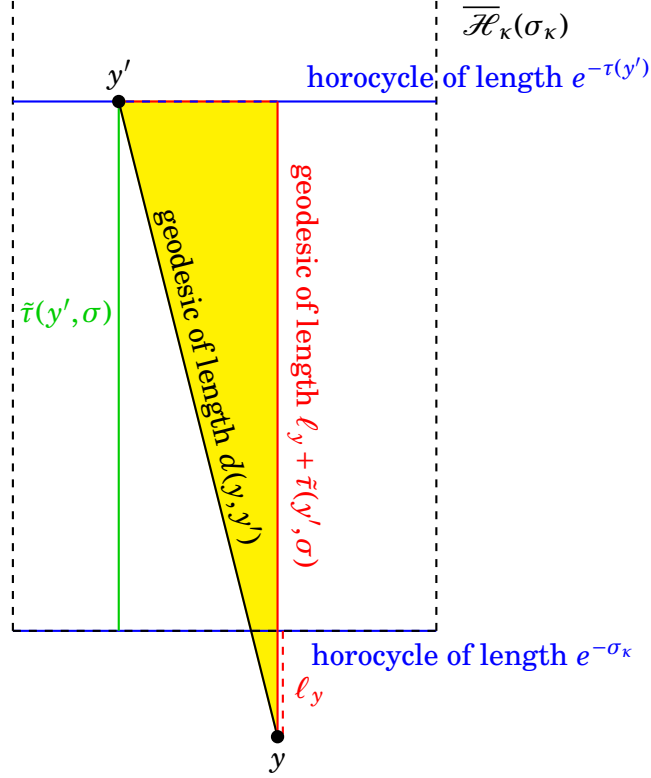

Theorem \ref{mainthm} is now a corollary of Theorem \ref{cusp-first-thm-multi}, Eq.~\eqref{h-limit}.

\begin{proof}[Proof of Theorem \ref{mainthm}]
Choose $\vecsigma\in\RR^K$ so that (i) its coefficients $\sigma_\kappa$ are sufficiently large as assumed above and (ii) the point $y\in\scrY\setminus\overline\scrH_\kappa^\circ(\sigma_\kappa)$ is equidistant to all cuspidal neighbourhoods $\overline\scrH_\kappa(\sigma_\kappa)$. That is, there is a $\ell_y\geq 0$ such that, for all $\kappa\in[K]$, we have
\begin{equation}
\ell_y = d_\scrY(y, \overline\scrH_\kappa(\sigma_\kappa)) .
\end{equation}
The triangle inequality (cf. Figure \ref{fig123}) then shows that for $y'\in \overline\scrH_\kappa(\sigma_\kappa)$ we have
\begin{equation}\label{tri}
\left| \widetilde\tau(y',\vecsigma) + \ell_y - d_\scrY(y,y') \right| \leq \e^{-\tau(y')}.
\end{equation}
Here the upper bound follows from the observation that the point $y'$ is at height $\tau(y')$ and therefore contained in a closed horocycle of length $\e^{-\tau(y')}$; that length in turn is greater than the geodesic distance between the end points of the corresponding horocycle segment. Since $\tau(y')\gg \log T \to\infty$ and by the continuity of the limit in $H$, Eq.~\eqref{h-limit} then implies that 
\begin{equation}\label{maineq234}
\lim_{T \to \infty} \lambda\left\{ x_0\in\scrX :  \sup_{0<s\leq T}  d_\scrY\left(y, \pi\circ h^+_s(x_0) \right) -\ell_y  >  H + \log T  \right\} 
=  \int_H^\infty \rho_\vecsigma(s) ds  
\end{equation}
and hence
\begin{equation}\label{maineq23456}
\lim_{T \to \infty} \lambda\left\{ x_0\in\scrX :  \sup_{0<s\leq T}  d_\scrY\left(y, \pi\circ h^+_s(x_0) \right)   >  H + \log T  \right\} 
=  \int_{H}^\infty \rho_\vecsigma(s-\ell_y) ds  .
\end{equation}
This completes the proof of Theorem \ref{mainthm} with the limit density given by $\omega_y = \rho_\vecsigma(s-\ell_y)$; the tail estimate follows from \eqref{5pt8multirho00} and \eqref{5pt8multirho}.
\end{proof}

We finally note that, in the case of the modular surface $\scrY$, we have $\omega_y(s)=\rho(s)$ if $y$ is a point on the (unique) closed horocycle of length one in $\scrY$.

\end{document}